\documentclass{cimart}

\newcommand{\Var}{\mathrm{Var}}
\newcommand{\I}{\mathrm{I}}

\newcommand{\Nov}{\mathrm{Nov}}
\newcommand{\Com}{\mathrm{Com}}
\newcommand{\T}{\mathrm{T}}
\newcommand{\As}{\mathrm{As}}

\newcommand{\Lie}{\mathrm{Lie}}
\newcommand{\Perm}{\mathrm{Perm}}

\newcommand{\Leib}{\mathrm{Leib}}

\newcommand{\Dend}{\mathrm{Dend}}

\newcommand{\pre}{\mathrm{pre}}

\newcommand{\di}{\mathrm{di}}

\newcommand{\LS}{\mathrm{LS}}

\newcommand{\Alt}{\mathrm{Alt}}
\newcommand{\Nilp}{\mathrm{Nilp}}
\newcommand{\R}{\mathrm{R}}

\let\<\langle
\let\>\rangle

\title{
    Varieties of initial dialgebras and some of their Koszul dual operads
    }

\authors{
    Aigerim Dauletiyarova, Abdenacer Makhlouf, and Bauyrzhan Sartayev
    }

\authorinfo[
    Aigerim Dauletiyarova]{
    Narxoz University,  Kazakhstan}{%
    d\_aigera95@mail.ru
    }

\authorinfo[
    Abdenacer Makhlouf]{
    University of Upper Alsace,  France}{%
    abdenacer.makhlouf@uha.fr
    }

\authorinfo[
    Bauyrzhan Sartayev]{
    Narxoz University,  Kazakhstan;  SDU University,  Kazakhstan}{%
    baurjai@gmail.com
    }

\abstract{%
In this paper, for a given variety $\Var$, we present a universal algorithm for constructing a subvariety of $\Var$-dialgebras from which one can recover an algebra belonging to $\Var$. Such a subvariety is called the variety of initial $\Var$-dialgebras. In addition, we construct a basis of the free initial Lie and associative dialgebras.
}



\VOLUME{35}
\YEAR{2027}
\ISSUE{2}
\NUMBER{2}
\DOI{https://doi.org/10.46298/cm.17559}
\licence{CC BY 4.0}
\editinfo{February 22, 2026}{May 3, 2026}{Ivan Kaygorodov and Maxime Fairon}

\acknowledgments{This research is funded by the Science Committee of the Ministry of Science and Higher Education of the Republic of Kazakhstan (Grant No. AP23484665).
The authors thank CIRM, where the first part of this work was prepared. We would like to thank the referees for their suggestions.}

\begin{document}

\section{Introduction}

Dialgebras were introduced by Loday as a two-operation counterpart of classical algebraic structures \cite{Loday}. 

In the theory of dialgebras, the first natural Lie-type objects are non-skew-symmetric analogues of Lie algebras, namely Leibniz algebras \cite{Blokh}. 
They are defined by the Leibniz identity
\[
(ab)c \;=\; a(bc)\;-\;b(ac).
\]
In the literature, Leibniz algebras are sometimes called a dialgebraic analogue of Lie algebras, since the skew-symmetry is dropped while the derivation property of the left multiplication is preserved \cite{LodPir}.

In operadic terms, associative dialgebras are algebras over the diassociative operad $\di\text{-}\As$, which is Koszul dual to the dendriform operad $\Dend$ \cite{LV}. 
This duality explains the splitting of identities and provides a convenient language for studying dialgebraic analogues of operadic constructions, including Manin products. Following this approach, Kolesnikov introduced a general procedure which assigns to a variety (operad) $\Var$ its dialgebraic analogue $\di\text{-}\Var$ and relates $\Var$-dialgebras to conformal and pseudo-algebraic realizations; in particular, $\Var$-dialgebras embed into suitable pseudo-algebras of type $\Var$ \cite{KolVar}.

In \cite{DongProp}, it was proved that the Dong property is preserved under the dialgebraic construction: if an operad $\Var$ satisfies the Dong property, then the operad $\di\text{-}\Var$ also satisfies the Dong property. 
In \cite{SartDzhMashNovikov}, an embedding of Novikov dialgebras into perm algebras endowed with a derivation was considered.

It is well known that the white Manin product \cite{GK94} yields the dialgebra version $\di\text{-}\Var$ of the original variety $\Var$; namely,
\[
\Perm\circ\Var=\di\text{-}\Var.
\]
Analogously, under the black Manin product, the following equality holds:
\[
\pre\text{-}\Lie \bullet \Var=\pre\text{-}\Var.
\]
Such relations lead to new interesting classes of algebras. For example, if
$\Var=\Nov$, then we obtain
\[
(\di\text{-}\Nov)^{(!)}=(\Perm\circ\Nov)^{(!)}=\pre\text{-}\Nov=\pre\text{-}\Lie\bullet\Nov,
\]
where $\Nov$ denotes the operad corresponding to the variety of Novikov algebras. For recent results on Novikov algebras, see \cite{DauSar, DotZhak, GaoGuoHanZhang, KolMashSar, LiHongQFNov}.

Another interesting result in this direction is that
\[
\dim(\di\text{-}\Var(n))=n\cdot k_n,
\]
for a given $\Var$ with $\dim(\Var(n))=k_n$, see \cite{Kol2017}. 
The same paper also provides an algorithm for constructing a basis of $\di\text{-}\Var\langle X\rangle$ from a given basis of $\Var\langle X\rangle$.

If we define a new operation $\star$ in $\pre\text{-}\Var\langle X\rangle$ by
\[
a \star b =\frac{1}{2}(a \vdash b + a \dashv b),
\]
then the resulting algebra $\pre\text{-}\Var^{(\star)}\langle X\rangle$ belongs to the variety $\Var$. However, we do not have an analogous general procedure for deriving the original algebra from its dialgebra version. For this reason, we propose an algorithm that helps to recover the original algebra from its dialgebra version.

It is known that not every algebra from $\pre\text{-}\Var^{(\star)}$ can be embedded into some algebra of $\pre\text{-}\Var$. As the first example, if $\Var=\Com$, then $\pre\text{-}\Com$ is the variety of Zinbiel algebras. In \cite{KolZinb}, it was given an example of an associative-commutative algebra that cannot be embedded into some appropriate algebra from the variety of Zinbiel algebras under the operation $\star$.

In this paper, we propose a universal algorithm that, for a given variety $\Var$, constructs a subvariety of $\Var$-dialgebras with the property that the $\star$-product recovers an algebra belonging to $\Var$. 
The main advantage of this construction is that, in the commutative case $\Var=\Com$, every associative-commutative algebra admits an embedding into a suitable algebra from this subvariety of $\di\text{-}\Com$, viewed with the operation $\star$. We also compute the Koszul dual operad of $\di$-$\Var^{\I}$ in several cases, where $\di$-$\Var^{\I}$ denotes the variety of initial $\Var$-dialgebras.
For $\Var=\Com$ and $\Lie$, we obtain the generalizations of the varieties of left-symmetric algebras and Zinbiel algebras. Moreover, if we define the operation $\star$ in $\di$-$\Var^{\I\;(!)}$, then as in the classical cases we obtain Lie and associative-commutative algebras. For recent results on Zinbiel algebras, we refer the reader to \cite{Zinb2, KolMashSar, Zinb3}.

Since varieties defined by polynomial identities of degrees $2$ and $3$ are in one-to-one correspondence with quadratic operads, we will use the same terminology for a variety $\Var$ and for the corresponding operad throughout the paper. 
We denote by $\Com$ the variety of associative-commutative algebras. 
All algebras are considered over a field $\mathbb{K}$ of characteristic $0$.

\section{Initial associative and pre-Lie dialgebras}

In this section, we present several concrete cases that yield the required subvarieties within the varieties of associative and pre-Lie dialgebras.

An algebra with one operation is called associative if it satisfies the identity
\[
(ab)c = a(bc).
\]

Every dialgebra satisfies $0$-identities, which are
\[
a\dashv(b\dashv c)=a\dashv(b\vdash c)
\]
and
\[
(a\vdash b)\vdash c=(a\dashv b)\vdash c.
\]

\begin{definition}
An associative dialgebra is a vector space with two bilinear operations satisfying the following additional identities:
\[
(a\dashv b)\dashv c=a\dashv(b\dashv c),
\]
\[
(a\vdash b)\dashv c=a\vdash(b\dashv c)
\]
and
\[
(a\vdash b)\vdash c=a\vdash(b\vdash c).
\]
\end{definition}

We say that an associative dialgebra is initial if, in addition, it satisfies the identity
\begin{equation}\label{as1}
(a\vdash b)\vdash c=a\dashv(b\dashv c).
\end{equation}
Let us denote by $\di$-$\As^{\I}$ and $\di$-$\As^{\I}\<X\>$ the variety of initial associative dialgebras and the corresponding free algebra from this variety, respectively.
\begin{proposition}
An algebra $\di$-$\As^{\I\;(\star)}\<X\>$ is associative.
\end{proposition}
\begin{proof}
It is enough to check that 
\[
(a\star b)\star c=a\star (b\star c).
\]
On the left and right sides, we obtain
\[
(a\dashv b)\dashv c+(a\vdash b)\dashv c+(a\dashv b)\vdash c+(a\vdash b)\vdash c
\]
and 
\[
a\dashv(b\dashv c)+a\dashv(b\vdash c)+a\vdash(b\dashv c)+a\vdash(b\vdash c),
\]
respectively. Using the defining identities of the associative dialgebra, we perform the same reductions and, in the end, obtain exactly identity~\eqref{as1}.
\end{proof}

\begin{definition}
An algebra is called pre-Lie (or, in other sources, left-symmetric algebras \cite{LS1,LS2,LS3}, 
or Vinberg–Koszul–Gerstenhaber algebras introduced in \cite{Gerst, Vin, k61}) if it satisfies the identity
\[
(ab)c-a(bc) = (ba)c-b(ac).
\]
\end{definition}
\begin{definition}
A left-symmetric dialgebra is a vector space with two operations satisfying the following additional identities:
\[
(a\vdash b)\dashv c-a\vdash(b\dashv c)=(b\dashv a)\dashv c-b\dashv(a\dashv c)
\]
and
\[
(a\vdash b)\vdash c-a\vdash(b\vdash c)=(b\vdash a)\vdash c-b\vdash(a\vdash c).
\]
\end{definition}

We call a left-symmetric dialgebra initial if it is a left-symmetric dialgebra with identity
\begin{equation}\label{ls1}
(a\vdash b)\vdash c-a\dashv(b\dashv c) = (b\vdash a)\vdash c-b\dashv(a\dashv c).    
\end{equation}
Let us denote by $\di$-$\LS^{\I}$ and $\di$-$\LS^{\I}\<X\>$ the variety of initial left-symmetric dialgebras and the corresponding free algebra from this variety, respectively.
\begin{proposition}
An algebra $\di$-$\LS^{\I\;(\star)}\<X\>$ is left-symmetric.
\end{proposition}
\begin{proof}
It is enough to check that 
\[
(a\star b)\star c-a\star (b\star c)=(b\star a)\star c-b\star (a\star c).
\]
On the left and right hand sides, we obtain
\begin{multline*}
(a\dashv b)\dashv c+(a\vdash b)\dashv c+(a\dashv b)\vdash c+(a\vdash b)\vdash c\\
-a\dashv(b\dashv c)-a\dashv(b\vdash c)-a\vdash(b\dashv c)-a\vdash(b\vdash c)    
\end{multline*}
and
\begin{multline*}
(b\dashv a)\dashv c+(b\vdash a)\dashv c+(b\dashv a)\vdash c+(b\vdash a)\vdash c\\
-b\dashv(a\dashv c)-b\dashv(a\vdash c)-b\vdash(a\dashv c)-b\vdash(a\vdash c).    
\end{multline*}
Using the defining identities of the associative dialgebra, we perform the same reductions, and at
the end, we obtain exactly the identity~\eqref{ls1}.
\end{proof}

\begin{remark}
It is straightforward to verify that the variety of initial associative dialgebras is different from the variety of dendriform algebras. Moreover, neither of these varieties is contained in the other. An analogous statement holds for the varieties of initial left-symmetric dialgebras and pre-left-symmetric algebras.
\end{remark}

\section{General construction of initial dialgebras}

In this section, we give an explicit algorithm which defines the variety of initial dialgebras from the variety of algebras $\Var$.

\begin{theorem}\label{initial}
Let $\Var$ be a variety of algebras defined by a single identity of the form
\begin{equation}\label{generalID}
    \sum_{\sigma \in \mathbb{S}_3} \alpha_\sigma \,(x_{\sigma(1)}x_{\sigma(2)})x_{\sigma(3)}
    +
    \sum_{\sigma \in \mathbb{S}_3} \beta_\sigma \,x_{\sigma(1)}(x_{\sigma(2)}x_{\sigma(3)}) \;=\; 0.
\end{equation}
\noindent Then for a $\Var$-dialgebra $A$, the algebra $A^{(\star)}$ belongs to $\Var$ if and only if $A$ satisfies the following additional identity:
\begin{equation}\label{initialID}
\sum_{\sigma \in \mathbb{S}_3} \alpha_\sigma \,(x_{\sigma(1)} \vdash x_{\sigma(2)}) \vdash x_{\sigma(3)}
    +
    \sum_{\sigma \in \mathbb{S}_3} \beta_\sigma \,x_{\sigma(1)} \dashv (x_{\sigma(2)} \dashv x_{\sigma(3)}) \;=\; 0.    
\end{equation}
\end{theorem}
\begin{proof}
The single defining identity of $\Var$ is
\[
\begin{aligned}
& \ \alpha_{123}(x_1 x_2)x_3
+\alpha_{132}(x_1 x_3)x_2
+\alpha_{213}(x_2 x_1)x_3\\
+& \ \alpha_{231}(x_2 x_3)x_1 
+\alpha_{312}(x_3 x_1)x_2
+\alpha_{321}(x_3 x_2)x_1 \\
+& \ \beta_{123} x_1(x_2 x_3)
+\beta_{132} x_1(x_3 x_2)
+\beta_{213} x_2(x_1 x_3)\\
+& \ \beta_{231} x_2(x_3 x_1)
+\beta_{312} x_3(x_1 x_2)
+\beta_{321} x_3(x_2 x_1)
= 0.
\end{aligned}
\]
Then $\Var$ dialgebra is defined by the following three additional identities:
\[
\begin{aligned}
& \ \alpha_{123}(x_1 \dashv x_2)\dashv x_3
+\alpha_{132}(x_1 \dashv x_3)\dashv x_2
+\alpha_{213}(x_2 \vdash x_1)\dashv x_3\\
+& \ \alpha_{231}(x_2 \vdash x_3)\vdash x_1
+\alpha_{312}(x_3 \vdash x_1)\dashv x_2
+\alpha_{321}(x_3 \vdash x_2)\vdash x_1\\
+& \ \beta_{123}\,x_1\dashv(x_2\vdash x_3)
+\beta_{132}\,x_1\dashv(x_3\vdash x_2)
+\beta_{213}\,x_2\vdash(x_1\dashv x_3)\\
+& \ \beta_{231}\,x_2\vdash(x_3\vdash x_1)
+\beta_{312}\,x_3\vdash(x_1\dashv x_2)
+\beta_{321}\,x_3\vdash(x_2\vdash x_1)=0,
\end{aligned}
\]
\[
\begin{aligned}
& \ \alpha_{123}(x_1 \vdash x_2)\dashv x_3
+\alpha_{132}(x_1 \vdash x_3)\vdash x_2
+\alpha_{213}(x_2 \dashv x_1)\dashv x_3\\
+& \ \alpha_{231}(x_2 \dashv x_3)\dashv x_1
+\alpha_{312}(x_3 \vdash x_1)\vdash x_2
+\alpha_{321}(x_3 \vdash x_2)\dashv x_1\\
+& \ \beta_{123}\,x_1\vdash(x_2\dashv x_3)
+\beta_{132}\,x_1\vdash(x_3\vdash x_2)
+\beta_{213}\,x_2\dashv(x_1\vdash x_3)\\
+& \ \beta_{231}\,x_2\dashv(x_3\vdash x_1)
+\beta_{312}\,x_3\vdash(x_1\vdash x_2)
+\beta_{321}\,x_3\vdash(x_2\dashv x_1)=0
\end{aligned}
\]
and
\[
\begin{aligned}
& \ \alpha_{123}(x_1 \vdash x_2)\vdash x_3
+\alpha_{132}(x_1 \vdash x_3)\dashv x_2
+\alpha_{213}(x_2 \vdash x_1)\vdash x_3\\
+& \ \alpha_{231}(x_2 \vdash x_3)\dashv x_1
+\alpha_{312}(x_3 \dashv x_1)\dashv x_2
+\alpha_{321}(x_3 \dashv x_2)\dashv x_1\\
+& \ \beta_{123}\,x_1\vdash(x_2\vdash x_3)
+\beta_{132}\,x_1\vdash(x_3\dashv x_2)
+\beta_{213}\,x_2\vdash(x_1\vdash x_3)\\
+& \ \beta_{231}\,x_2\vdash(x_3\dashv x_1)
+\beta_{312}\,x_3\dashv(x_1\vdash x_2)
+\beta_{321}\,x_3\dashv(x_2\vdash x_1)=0.
\end{aligned}
\]
We show that the identity
\[
\begin{aligned}
& \ \alpha_{123}(x_1\star x_2)\star x_3
+\alpha_{132}(x_1\star x_3)\star x_2
+\alpha_{213}(x_2\star x_1)\star x_3 
+\alpha_{231}(x_2\star x_3)\star x_1 \\
+& \ \alpha_{312}(x_3\star x_1)\star x_2
+\alpha_{321}(x_3\star x_2)\star x_1 
+\beta_{123} x_1\star(x_2\star x_3)
+\beta_{132} x_1\star(x_3\star x_2)\\
+& \ \beta_{213} x_2\star(x_1\star x_3)
+\beta_{231} x_2\star(x_3\star x_1)
+\beta_{312} x_3\star(x_1\star x_2)
+\beta_{321} x_3\star(x_2\star x_1)
= 0.
\end{aligned}
\]
modulo the identities of $\Var$ dialgebra gives \eqref{initialID}. By $a\star b=a\vdash b+a\dashv b$, we have
\[
\begin{aligned}
0=& \ \alpha_{123}\Bigl(
  (x_1\vdash x_2)\vdash x_3
 +(x_1\dashv x_2)\vdash x_3
 +(x_1\vdash x_2)\dashv x_3
 +(x_1\dashv x_2)\dashv x_3
\Bigr)\\
+& \ \alpha_{132}\Bigl(
  (x_1\vdash x_3)\vdash x_2
 +(x_1\dashv x_3)\vdash x_2
 +(x_1\vdash x_3)\dashv x_2
 +(x_1\dashv x_3)\dashv x_2
\Bigr)\\
+& \ \alpha_{213}\Bigl(
  (x_2\vdash x_1)\vdash x_3
 +(x_2\dashv x_1)\vdash x_3
 +(x_2\vdash x_1)\dashv x_3
 +(x_2\dashv x_1)\dashv x_3
\Bigr)\\
+& \ \alpha_{231}\Bigl(
  (x_2\vdash x_3)\vdash x_1
 +(x_2\dashv x_3)\vdash x_1
 +(x_2\vdash x_3)\dashv x_1
 +(x_2\dashv x_3)\dashv x_1
\Bigr)\\
+& \ \alpha_{312}\Bigl(
  (x_3\vdash x_1)\vdash x_2
 +(x_3\dashv x_1)\vdash x_2
 +(x_3\vdash x_1)\dashv x_2
 +(x_3\dashv x_1)\dashv x_2
\Bigr)\\
+& \ \alpha_{321}\Bigl(
  (x_3\vdash x_2)\vdash x_1
 +(x_3\dashv x_2)\vdash x_1
 +(x_3\vdash x_2)\dashv x_1
 +(x_3\dashv x_2)\dashv x_1
\Bigr)\\
+& \ \beta_{123}\Bigl(
  x_1\vdash(x_2\vdash x_3)
 +x_1\vdash(x_2\dashv x_3)
 +x_1\dashv(x_2\vdash x_3)
 +x_1\dashv(x_2\dashv x_3)
\Bigr)\\
\end{aligned}
\]
\[
\begin{aligned}
+& \ \beta_{132}\Bigl(
  x_1\vdash(x_3\vdash x_2)
 +x_1\vdash(x_3\dashv x_2)
 +x_1\dashv(x_3\vdash x_2)
 +x_1\dashv(x_3\dashv x_2)
\Bigr)\\
+& \ \beta_{213}\Bigl(
  x_2\vdash(x_1\vdash x_3)
 +x_2\vdash(x_1\dashv x_3)
 +x_2\dashv(x_1\vdash x_3)
 +x_2\dashv(x_1\dashv x_3)
\Bigr)\\
+& \ \beta_{231}\Bigl(
  x_2\vdash(x_3\vdash x_1)
 +x_2\vdash(x_3\dashv x_1)
 +x_2\dashv(x_3\vdash x_1)
 +x_2\dashv(x_3\dashv x_1)
\Bigr)\\
+& \ \beta_{312}\Bigl(
  x_3\vdash(x_1\vdash x_2)
 +x_3\vdash(x_1\dashv x_2)
 +x_3\dashv(x_1\vdash x_2)
 +x_3\dashv(x_1\dashv x_2)
\Bigr)\\
+& \ \beta_{321}\Bigl(
  x_3\vdash(x_2\vdash x_1)
 +x_3\vdash(x_2\dashv x_1)
 +x_3\dashv(x_2\vdash x_1)
 +x_3\dashv(x_2\dashv x_1)
\Bigr).\\
\end{aligned}
\]
Using the defining identities of the $\Var$-dialgebra, we directly obtain the identity \eqref{initialID}.
\end{proof}

Theorem \ref{initial} motivates the following definition.

\begin{definition}\label{definitial}
Let $\Var$ be a variety of algebras defined by a single identity of the form \eqref{generalID}. Then the variety of initial $\Var$-dialgebras is obtained from the variety of $\Var$-dialgebras by adjoining an additional identity of the form \eqref{initialID}.
\end{definition}

\begin{remark}
In fact, Theorem~\ref{initial} and Definition~\ref{definitial} extend naturally to varieties
defined by several identities. More precisely, suppose that a variety $\Var$ is defined
by a family of identities
\[
f_i(x_1,x_2,x_3)=0,\qquad i\in I,
\]
where each $f_i$ is of the form \eqref{generalID}. For every such identity, we consider
the corresponding dialgebra identity $f_i^{\I}=0$ obtained from \eqref{initialID}.
Then the variety of initial $\Var$-dialgebras is defined as the subvariety of
$\Var$-dialgebras satisfying all identities $f_i^{\I}=0$, $i\in I$.
In particular, if $A$ belongs to this variety, then the algebra $A^{(\star)}$ belongs
to $\Var$.
\end{remark}

\begin{example}
Consider the variety $\Var = \Com$ of commutative-associative algebras. In this case
$\di$-$\Com = \Perm$, the variety of perm algebras defined by associativity and
right-commutativity. The initial commutative dialgebra then coincides with a perm algebra
satisfying the additional identity
\[
  (a \vdash b) \vdash c = a \dashv (b \dashv c).
\]
Using commutativity $a\vdash b=b\dashv a$ in $\di$-$\Com$, it follows that $\di$-$\Com^{\I}$ is a two-sided perm algebra, that is, an associative algebra satisfying the identities of left- and right-commutativity.
For further details on perm algebras, see \cite{perm2,binaryperm}.
\end{example}

\begin{example}
Now let $\Var = \Lie$. It is well known that $\di$-$\Lie = \Leib$, where $\Leib$ denotes
the variety of Leibniz algebras. In this situation, the initial Lie dialgebra coincides
with a Leibniz algebra satisfying the additional identity
\[
  a\vdash (b\vdash c)-a\vdash (c\vdash b)-b\vdash (a\vdash c)+b\vdash (c\vdash a)
  +c\vdash (a\vdash b)-c\vdash (b\vdash a)=0.
\]
\end{example}

\begin{example}
Let $\Var=\Alt$ be the variety of alternative algebras. Then $\di$-$\Alt^{\I}$ is the subvariety of $\di$-$\Alt$ defined by the additional identities
\[
(a\vdash b)\vdash c-a\dashv(b \dashv c)+(a\vdash c)\vdash b-a\dashv(c \dashv b)=0
\]
and
\[
(a\vdash b)\vdash c-a\dashv(b\dashv c)+(b\vdash a)\vdash c-b\dashv(a\dashv c)=0.
\]
\end{example}

\begin{remark}
It is tempting to expect an equality of operads
\[
\T\text{-}\Perm \circ \Var \;=\; \di\text{-}\Var^{\I},
\]
where $\T\text{-}\Perm$ denotes the variety of two-sided perm algebras and $\circ$ is the white Manin product of two quadratic operads. 
In general, this equality fails. 

Indeed, take $\Var=\Lie$. Then the operad $\T\text{-}\Perm\circ\Lie$ yields $\di\text{-}\Lie^{\I}$ together with an additional identity, namely
\[
a\vdash(b\vdash c)+(b\vdash c)\vdash a=0.
\]

Similarly, one may ask whether
\[
\di\text{-}\Lie^{\I}\bullet\Var=\di\text{-}\Var^{\I}
\]
holds. However, this equality fails for \(\Var=\R\textrm{-}\Nilp\), where \(\R\textrm{-}\Nilp\) denotes the right-nilpotent quadratic operad. To compute \(\di\text{-}\Lie^{\I}\bullet\R\textrm{-}\Nilp\), we use \cite{DongProp}.

More precisely, an algebra \(A\) with a bilinear operation \(\cdot\) is a \((\mathcal{P}\bullet\mathcal{Q})\)-algebra if and only if, for every \(\mathcal{P}^{!}\)-algebra \(B\), the space \(B\otimes A\), equipped with the operation
\[
(x\otimes a)\cdot(y\otimes b)=xy\otimes(a\dashv b)+yx\otimes(a\vdash b),
\qquad x,y\in B,\ a,b\in A,
\]
is a \(\mathcal{Q}\)-algebra.

Now let \(\mathcal{P}=\R\textrm{-}\Nilp\) and \(\mathcal{Q}=\di\text{-}\Lie^{\I}\). Then
\[
\mathcal{P}^{!}=\l\textrm{-}\Nilp,
\]
where \(\l\textrm{-}\Nilp\) is the left-nilpotent quadratic operad. Hence, the Leibniz identity
\[
((x\otimes a)(y\otimes b))(z\otimes c)
-(x\otimes a)((y\otimes b)(z\otimes c))
+(y\otimes b)((x\otimes a)(z\otimes c))=0
\]
holds on \(B\otimes A\) if and only if the following identities are satisfied:
\[
(a\vdash b)\dashv c=(a\dashv b)\dashv c=a\vdash(b\vdash c)=a\vdash(b\dashv c)=0.
\]
Moreover, all identities obtained from
\[
a(bc)-a(cb)-b(ac)+b(ca)+c(ab)-c(ba)=0
\]
follow from the identities above.

By contrast, by definition, the operad \(\di\text{-}\R\textrm{-}\Nilp^{\I}\) is defined by the identities
\[
a\dashv(b\dashv c)=a\dashv(b\vdash c)=a\vdash(b\vdash c)=a\vdash(b\dashv c)=0
\]
together with
\[
(a\vdash b)\vdash c=(a\dashv b)\vdash c.
\]
Therefore, the two operads do not coincide, since the last identity does not appear in the black Manin product.
\end{remark}

\section{The free initial Lie and associative dialgebras}

In this section, we construct a linear basis of the free initial Lie dialgebra and the free initial associative dialgebra.
Before stating the results, we fix the standard notation
\[
[a,b]=ab-ba,\qquad \{a,b\}=ab+ba,
\]
where $[a,b]$ and $\{a,b\}$ denote the commutator and the anti-commutator, respectively.
\begin{lemma}\label{lemma1}
In the algebra $\di$-$\Lie^{\I}\<X\>$ the following identities hold:
\begin{align}
&[\{a,b\},c]= -\,\{\{a,b\},c\},\label{1}\\
&\{a,[b,c]\}=\{\{a,c\},b\} - \{\{a,b\},c\},\label{2}\\
&[[a,b],c]+[[b,c],a]+[[c,a],b]=0,\label{3}\\
&\{\{\{a,b\},d\},c\}= \{\{\{a,b\},c\},d\},\label{4}\\
&\{\{a,b\},\{c,d\}\}=0.\label{5}
\end{align}
\end{lemma}
\begin{proof}
Identity \eqref{5} was established in \cite{Serdica}. Some of these identities follow immediately from the relation
\[
\{a,b\}c=0,
\]
which holds in every Leibniz algebra.
The remaining identities can be verified by means of computer algebra computations, for example, with the software package Albert\footnote{Albert version 4.0m6. Software.}.
\end{proof}

By Lemma~\ref{lemma1}, we obtain the following description of the polarization of 
the algebra $\di\text{-}\Lie^{\I}\langle X\rangle$.

\begin{proposition}
The polarization of the algebra $\di\text{-}\Lie^{\I}\langle X\rangle$ is given exactly
by identities~\eqref{1}, \eqref{2} and~\eqref{3}.
\end{proposition}

Let us define the set $\mathcal{C}^S_n$ as follows:
\[
\mathcal{C}^S_n=\big\{\; \{\{\cdots\{\{x_{i_1},x_{i_2}\},x_{i_3}\},\cdots\},x_{i_n}\}|\;i_1\leq i_2,\;i_3\leq\ldots\leq i_n \;\big\},
\]
Also, we set
\[
\mathcal{L}^I=\bigcup_{i=1} \mathcal{C}_i^S \cup \mathcal{L},
\]
where $\mathcal{L}$ is a basis of the free Lie algebra $\Lie\<X\>$.

\begin{theorem}\label{basis1}
The basis of the free initial Lie dialgebra is the set $\mathcal{L}^I$. In particular,
\[
\di\text{-}\Lie^{\I}\<X\>=\Lie\<X\>\;\oplus \Com^S\<X\>,
\]
as a vector space, where $\Com^S\<X\>$ is a free commutative algebra with identities \eqref{4} and \eqref{5}.
\end{theorem}
\begin{proof}
First, let us show that the set $\mathcal{L}^{\I}$ spans $\di\text{-}\Lie^{\I}\langle X\rangle$.
We rewrite each product in terms of the symmetric and anti-symmetric operations as
\[
  ab = \frac{1}{2}\bigl([a,b] + \{a,b\}\bigr).
\]
The identities \eqref{1} and \eqref{2} allow us to rewrite every monomial of
$\di\text{-}\Lie^{\I}\langle X\rangle$ as a sum of two subspaces, where one involves only
pure commutator operations and the other involves only pure anti-commutator operations.
By \eqref{3}, all commutator monomials can be written as a linear combination of the
elements of $\mathcal{L}$.

The identity \eqref{5} reduces all monomials of the form
\[
  \{\{A_1,A_2\},\{A_3,A_4\}\},
\]
so that we are left only with left-normed monomials of the form
\[
  \{\{\cdots\{\{x_{i_1},x_{i_2}\},x_{i_3}\},\dots\},x_{i_n}\}.
\]
Finally, commutativity together with \eqref{4} allows us to order the generators as
$x_{i_1} \le x_{i_2}$ and $x_{i_3} \le \dots \le x_{i_n}$.

Let us consider an algebra $\mathcal{A}\langle X\rangle$ with basis $\mathcal{L}^{\I}$ and
define the multiplication on $\mathcal{A}\langle X\rangle$ as follows:
\begin{itemize}
    \item $[[L],[L]]$ is defined as in the free Lie algebra;
    
    \item $\big\{\{\{\cdots\{\{x_{i_1},x_{i_2}\},x_{i_3}\},\dots\},x_{i_n}\},
      \{\{\cdots\{\{x_{j_1},x_{j_2}\},x_{j_3}\},\dots\},x_{j_m}\}\big\}=0$;
    
    \item $\big\{\{\{\cdots\{\{x_{i_1},x_{i_2}\},x_{i_3}\},\dots\},x_{i_n}\},x_t\big\}
      =
      \{\{\cdots\{\{x_{i_1},x_{i_2}\},x_{i_3}\},\ldots\},x_t\},\dots\},x_{i_n}\}$;
    
    \item $\big[\{\{\cdots\{\{x_{i_1},x_{i_2}\},x_{i_3}\},\dots\},x_{i_n}\},
      \{\{\cdots\{\{x_{j_1},x_{j_2}\},x_{j_3}\},\dots\},x_{j_m}\}\big]=0$;
    
    \item $\big[\{\{\cdots\{\{x_{i_1},x_{i_2}\},x_{i_3}\},\dots\},x_{i_n}\},x_t\big]
      =
      -\{\{\cdots\{\{x_{i_1},x_{i_2}\},x_{i_3}\},\ldots\},x_t\},\dots\},x_{i_n}\}$;
      
    \item $\big\{\{\{\cdots\{\{x_{i_1},x_{i_2}\},x_{i_3}\},\dots\},x_{i_n}\},[L]\big\}=0$;
    
    \item $\big[\{\{\cdots\{\{x_{i_1},x_{i_2}\},x_{i_3}\},\dots\},x_{i_n}\},[L]\big]=0$;
    
    \item $\big\{[L],[L]\big\} =
      \big\{
        [[\cdots[[x_{i_1},x_{i_2}],x_{i_3}],\dots],x_{i_n}],
        [[\cdots[[x_{j_1},x_{j_2}],x_{j_3}],\dots],x_{j_m}]
      \big\}\\
     =\{\{\cdots\{\{x_{i_1},x_{j_1}\},x_{k_3}\},\dots\},x_{k_{m+n}}\}
      +\{\{\cdots\{\{x_{i_2},x_{j_2}\},x_{l_3}\},\dots\},x_{l_{m+n}}\}\\
      -\{\{\cdots\{\{x_{i_1},x_{j_2}\},x_{s_3}\},\dots\},x_{s_{m+n}}\}
      -\{\{\cdots\{\{x_{i_2},x_{j_1}\},x_{r_3}\},\dots\},x_{r_{m+n}}\}$,
\end{itemize}
where $[L]$ is a left-normed pure Lie monomial and
\begin{align*}
\{k_3,\ldots,k_{m+n}\}= \ &\{i_2,\ldots,i_n,j_2,\ldots,j_m\},\\
\{l_3,\ldots,l_{m+n}\}= \ &\{i_1,i_3,\ldots,i_n,j_1,j_3,\ldots,j_m\},\\
\{s_3,\ldots,s_{m+n}\}= \ &\{i_2,\ldots,i_n,j_1,j_3,\ldots,j_m\},\\
\{r_3,\ldots,r_{m+n}\}= \ &\{i_1,i_3,\ldots,i_n,j_2,\ldots,j_m\},
\end{align*}
with the indices ordered so that
\[
k_3\leq\cdots\leq k_{m+n},\quad
l_3\leq\cdots\leq l_{m+n},\quad
s_3\leq\cdots\leq s_{m+n},\quad
r_3\leq\cdots\leq r_{m+n}.
\]
By a straightforward computation, one verifies that the algebra $\mathcal{A}\langle X\rangle$
belongs to $\di$-$\Lie^{\I}$, i.e., satisfies the identities ~\eqref{1}, \eqref{2} and~\eqref{3}. It remains to observe that
\[
\mathcal{A}\langle X\rangle \cong \di\text{-}\Lie^{\I}\langle X\rangle,
\]
which completes the proof.
\end{proof}

\begin{corollary}
From Theorem \ref{basis1}, we obtain
\[
\dim(\di\text{-}\Lie^{\I}(n))=(n-1)!+\frac{n(n-1)}{2}.
\]
In particular, we have
\begin{center}
\begin{tabular}{c|ccccccc}
 $n$ & 1 & 2 & 3 & 4 & 5 & 6 & 7 \\
 \hline $\dim$$(\di$-$\Lie^{\I}$$(n))$ & 1 & 2 & 5 & 12 & 34 & 135 & 741
\end{tabular}
\end{center}
\end{corollary}

Now, we study a free initial associative dialgebra $\di\text{-}\As^{\I}\<X\>$. 
\begin{lemma}
An algebra $\di\text{-}\As^{\I}\<X\>$ satisfies the following identity:
\begin{equation}\label{idAsI}
((a\vdash b)\dashv c)\dashv d=((a\dashv b)\dashv c)\dashv d.
\end{equation}
\end{lemma}
\begin{proof}
It suffices to consider the composition of the identity
\[
(a\vdash b)\vdash c \;=\; a\dashv(b\dashv c)
\]
with itself. On the one hand, we obtain
\[
\begin{aligned}
a\vdash\bigl(b\vdash(c\vdash d)\bigr)
&=(a\dashv b)\dashv(c\vdash d)
=(a\dashv b)\dashv(c\dashv d)
=\bigl((a\dashv b)\dashv c\bigr)\dashv d.
\end{aligned}
\]
On the other hand,
\[
\begin{aligned}
a\vdash\bigl(b\vdash(c\vdash d)\bigr)
&=a\vdash\bigl(b\dashv(c\dashv d)\bigr)
=(a\vdash b)\dashv(c\dashv d)
=\bigl((a\vdash b)\dashv c\bigr)\dashv d,
\end{aligned}
\]
and the proof is completed.
\end{proof}

A linear basis of the algebra $\di\text{-}\As\<X\>$ is described in \cite{Loday} and consists of the monomials
\[
x_{i_1}\vdash\cdots\vdash x_{i_{k-1}}\vdash x_{i_k}\dashv x_{i_{k+1}}\dashv\cdots\dashv x_{i_n},
\]
where $k$ is any integer with $1\le k\le n$. In particular, this yields
\[
\dim\bigl(\di\text{-}\As(n)\bigr)=n\cdot n!.
\]

Since the operad derived from the associative dialgebras is nonsymmetric, the operad derived from the free initial associative dialgebra is nonsymmetric as well. Therefore, it suffices to determine a monomial basis of the free algebra $\di\text{-}\As^{\I}\<X\>$ on one generator.

Let us define the set $\mathcal{A}_i$ as follows:
\[
\mathcal{A}_1=\{x\},\; \mathcal{A}_2=\{x\vdash x,\; x\dashv x\},\;\mathcal{A}_3=\{(x\dashv x)\dashv x,\; (x\vdash x)\dashv x\},
\]
\[
\mathcal{A}_n=\underset{n}{\underbrace{x\dashv x \dashv \cdots \dashv x }},
\]
where $n\geq 4$. Also, we set
\[
\mathcal{A}^I=\bigcup_i \mathcal{A}_i.
\]

\begin{theorem}\label{basisDiAsI}
The set $\mathcal{A}^I$ is the basis of the algebra $\di\text{-}\As^{\I}\<x\>$.
\end{theorem}
\begin{proof}
First, we prove that the set $\mathcal{A}^{I}$ spans the algebra $\di\text{-}\As^{\I}\<x\>$. In degrees $\le 3$ this is immediate from the defining identities of $\di\text{-}\As^{\I}\<x\>$. Starting from degree $4$, we use the standard monomial basis of $\di\text{-}\As\<x\>$. Every monomial in $\di\text{-}\As^{\I}\<x\>$ can be rewritten in the form
\[
x\vdash\cdots\vdash x\vdash x\dashv x\dashv\cdots\dashv x .
\]
By \eqref{idAsI}, we have
\[
x\vdash x\dashv x\dashv x \;=\; x\dashv x\dashv x\dashv x .
\]
This identity allows one to move the symbol $\vdash$ one position to the left across a block of $\dashv$'s. Consequently,
\begin{multline*}
x\vdash\cdots\vdash x\vdash x\dashv x\dashv\cdots\dashv x
= x\vdash\cdots\vdash x\dashv x\dashv x\dashv\cdots\dashv x \\
= x\vdash\cdots\dashv x\dashv x\dashv x\dashv\cdots\dashv x
= \cdots
= x\dashv\cdots\dashv x\dashv x\dashv x\dashv\cdots\dashv x.
\end{multline*}

For monomials of the form
\[
x\vdash\cdots\vdash x\vdash x\vdash x\vdash x,\;x\vdash\cdots\vdash x\vdash x\vdash x\dashv x,
\]
it is enough to use the defining identities of $\di$-$\As^{\I}\<X\>$ and the previous rule.

Let us consider an algebra $\mathcal{A}\<x\>$ with a basis $\mathcal{A}^I$. Up to degree $3$, we define multiplication in $\mathcal{A}\<x\>$ that is consistent with the defining identities of $\di\text{-}\As^{\I}\<x\>$. Starting from degree $4$, the multiplication is defined as follows:
\begin{itemize}
    \item $(x\dashv\cdots\dashv x\dashv x)\vdash(x\dashv\cdots\dashv x\dashv x)=x\dashv\cdots\dashv x\dashv x$;
    \item $(x\dashv\cdots\dashv x\dashv x)\dashv(x\dashv\cdots\dashv x\dashv x)=x\dashv\cdots\dashv x\dashv x$.
\end{itemize}
By a straightforward computation, one verifies that the algebra $\mathcal{A}\langle x\rangle$
belongs to $\di$-$\As^{\I}$. It remains to note that
\[
\mathcal{A}\langle x\rangle \cong \di\text{-}\As^{\I}\langle x\rangle,
\]
which completes the proof.
\end{proof}

Since the operad $\di$-$\As^{\I}$ is nonsymmetric, to obtain a basis of $\di$-$\As^{\I}\<X\>$ algebra, we place all possible permutations of the alphabet $X$ in monomials of the set $\mathcal{A}^{I}$.

From the Theorem \ref{basisDiAsI}, we obtain
\[
\dim(\di\text{-}\As^{\I}(1))=1,\; \dim(\di\text{-}\As^{\I}(2))=4,\;
\dim(\di\text{-}\As^{\I}(3))=12,
\]
and starting from degree $4$, we have
\[
\dim(\di\text{-}\As^{\I}(n))=n!.
\]

\begin{remark}
There is no need to describe a basis of the free initial associative-commutative dialgebra, since it is a two-sided perm algebra. It is well known that a free two-sided perm algebra coincides with the free associative-commutative algebra in all homogeneous components of
degree at least~$3$.
\end{remark}

At first sight, one may expect that, from some degree on, the free algebra $\di\text{-}\Var^{\I}\langle X\rangle$ behaves like $\Var\langle X\rangle$, as happens for $\Var=\Com$ and $\Var=\As$. For $\Var=\Com$ and $\Var=\As$, it starts from degrees $3$ and $4$, respectively.
However, for $\Var=\Lie$, we obtain a completely different picture. 
Moreover, for $\Var=\LS$, the description of the free algebra $\di\text{-}\LS^{\I}\langle X\rangle$ becomes a nontrivial problem, where $\LS$ denotes the variety of left-symmetric algebras. By using the computer algebra as \cite{DotsHij}, we obtain
\begin{center}
\begin{tabular}{c|cccccc}
 $n$ & 1 & 2 & 3 & 4 & 5 \\
 \hline $\dim(\di\text{-}\LS^{\I}(n))$ & 1 & 4 &  24 & 176 & 1620
\end{tabular}
\end{center}
As the above examples indicate, there is no general relation between the dimensions of the free algebras $\Var\langle X\rangle$ and $\di\text{-}\Var^{\I}\langle X\rangle$.

\section{Some Koszul dual operads of operads $\di$-$\Var^{\I}$}

In this section, we compute the Koszul dual operads of $\di$-$\Com^{\I}$, $\di$-$\Lie^{\I}$, and $\di$-$\As^{\I}$. The motivation for this is the following diagram:
\begin{center}
\setlength{\unitlength}{1pt}
\begin{picture}(400,90)
  \put(115,75){$\di\text{-}\Var$}
  \put(240,75){$\pre\text{-}\Var$}
  \put(125,57){$\cup$}
  \put(250,57){$\cap$}
  \put(115,40){$\di\text{-}\Var^{\I}$}
  \put(240,40){$\di\text{-}\Var^{\I\,(!)}$}
  \put(175,43){\vector(1,0){55}}
  \put(210,43){\vector(-1,0){55}}
  \put(175,79){\vector(1,0){55}}
  \put(210,79){\vector(-1,0){55}}
\end{picture}
\end{center}
\vspace*{-\baselineskip}
\vspace*{-\baselineskip}

To compute the Koszul dual operad of $\di\text{-}\Var^{\I}$, we use the Lie-admissibility condition for $S\otimes U$, where $S$ is an algebra from $\di\text{-}\Var^{\I}$. If $\Var=\Com$, then we have
\begin{align*}
[[a\otimes u,b\otimes v],c\otimes w]= \ &(ab)c\otimes (uv)w-(ba)c\otimes (vu)w-
c(ab)\otimes w(uv)+c(ba)\otimes w(vu)\\
= \ &abc\otimes (uv)w-abc\otimes (vu)w-
abc\otimes w(uv)+abc\otimes w(vu),\\
[[b\otimes v,c\otimes w],a\otimes u]= \ &(bc)a\otimes (vw)u-(cb)a\otimes (wv)u-
a(bc)\otimes u(vw)+a(cb)\otimes u(wv)\\
= \ & abc\otimes (vw)u-abc\otimes (wv)u-
abc\otimes u(vw)+abc\otimes u(wv),\\
[[c\otimes w,a\otimes u],b\otimes v]= \ &(ca)b\otimes (wu)v-(ac)b\otimes (uw)v-
b(ca)\otimes v(wu)+b(ac)\otimes v(uw)\\
= \ &abc\otimes (wu)v-abc\otimes (uw)v-
abc\otimes v(wu)+abc\otimes v(uw).
\end{align*}
The sum of the same elements on the left side of the tensors yields the following result: 
\begin{proposition}
An operad $\di\text{-}\Com^{\I\;(!)}$ is the Lie-admissible operad. 
\end{proposition}

If $\Var=\Lie$, then we have
\begin{align*}
[[a\otimes u,b\otimes v],c\otimes w]= \ & (ab)c\otimes (uv)w-(ba)c\otimes (vu)w-
c(ab)\otimes w(uv)+c(ba)\otimes w(vu)\\
= \ & (a(cb)-b(ca)-c(ab)+c(ba))\otimes (uv)w\\
- \ & (-a(cb)+b(ca)+c(ab)-c(ba))\otimes (vu)w \\
- \ & c(ab)\otimes w(uv)+c(ba)\otimes w(vu)\\
[[b\otimes v,c\otimes w],a\otimes u]= \ & (bc)a\otimes (vw)u-(cb)a\otimes (wv)u-
a(bc)\otimes u(vw)+a(cb)\otimes u(wv)\\
= \ & (b(ca)-c(ba))\otimes (vw)u-(c(ba)-b(ca))\otimes (wv)u\\
- \ & (a(cb)+b(ac)-b(ca)-c(ab)+c(ba))\otimes u(vw)+a(cb)\otimes u(wv)\\
[[c\otimes w,a\otimes u],b\otimes v]= \ & (ca)b\otimes (wu)v-(ac)b\otimes (uw)v-b(ca)\otimes v(wu)+b(ac)\otimes v(uw)\\
= \ & (c(ab)-a(cb))\otimes (wu)v-(a(cb)-c(ab))\otimes (uw)v\\
- \ & b(ca)\otimes v(wu)+b(ac)\otimes v(uw).
\end{align*}
The sum of the same elements on the left side of the tensors yields the following result: 
\begin{proposition}
An operad $\di\text{-}\Lie^{\I\;(!)}$ is the left-commutative operad with the following identity:
\begin{equation}\label{Id52}
(uv)w+(vu)w-u(vw)+u(wv)-(wu)v-(uw)v=0.
\end{equation}
\end{proposition}

As we have seen, the algebras $\di\text{-}\Com^{\I\,(!)}\langle X\rangle$ and $\di\text{-}\Com^{(!)}\langle X\rangle$ yield Lie algebras when equipped with the commutator bracket. 
Let us now verify the following statement.

\begin{proposition}
The algebra $\di\text{-}\Lie^{\I\,(!)}\langle X\rangle$ as a Zinbiel algebra equipped with the anti-commutator product is associative-commutative.
\end{proposition}
\begin{proof}
By definition,
\[
\{\{a,b\},c\}=(ab)c+(ba)c+c(ab)+c(ba),
\]
and
\[
\{a,\{b,c\}\}=a(bc)+a(cb)+(bc)a+(cb)a.
\]
Using
\[
c(ab)=a(cb),\qquad c(ba)=b(ca),
\]
we obtain
\[
\{\{a,b\},c\}=(ab)c+(ba)c+a(cb)+b(ca).
\]
Therefore, the desired equality $\{\{a,b\},c\}=\{a,\{b,c\}\}$ is equivalent to
\[
(ab)c+(ba)c+b(ca)=a(bc)+(bc)a+(cb)a.
\]
By \eqref{Id52}, we obtain the result.
\end{proof}

As we noted before, not every associative-commutative algebra can be embedded into some algebra from the variety of Zinbiel algebras. However, the variety $\di\text{-}\Lie^{\I\;(!)}$ is a generalization of the variety of Zinbiel algebras, and it would be interesting to verify if it is possible to embed any associative-commutative algebra into some algebra from $\di\text{-}\Lie^{\I\;(!)}$.

From all verified observations, we can notice that in algebra $\di\text{-}\Var^{\I\;(!)}$ there can be defined an operation $\star$ to obtain an algebra from $\Var^{(!)}$.
\\

\noindent \textbf{Open problems:}

\begin{enumerate}

\item Construct the basis of the algebra $\di\text{-}\LS^{\I}\langle X\rangle$.

\item Find a universal algorithm that allows us to construct the basis of $\di\text{-}\Var^{\I}\langle X\rangle$ starting from the known basis of $\Var\<X\>$.

\item Prove that
\[
\Var\hookrightarrow \di\text{-}\Var^{\I},
\]
i.e., is it possible to embed any algebra from $\Var$ into an appropriate algebra from $\di\text{-}\Var^{\I}$ under the operation $\star$?

\item Using the defining identities of $\pre$-$\Var$, find a universal algorithm for constructing $\di\text{-}\Var^{\I\;(!)}$.

\item If an operad $\Var$ is Koszul, does it follow that the operad $\di\text{-}\Var^{\I\,(!)}$ is also Koszul? In the classical case, from the Koszulness of $\Var$, it follows that $\di$-$\Var$ is also a Koszul operad. 

\item Find a criterion that identifies the elements of $\di\text{-}\Var^{\I\;(\star)}\<X\>$ in $\di\text{-}\Var^{\I}\<X\>$.

\item Is it possible to embed any associative-commutative algebra into some algebra from $\di\text{-}\Lie^{\I\;(!)}$ under the anti-commutator operation?

\item Prove that
\[
\Var^{(!)}\hookrightarrow \di\text{-}\Var^{\I\;(!)},
\]
i.e., is it possible to embed any algebra from $\Var^{(!)}$ into an appropriate algebra from $\di\text{-}\Var^{\I\;(!)}$ under the operation $\star$?
\end{enumerate}


\end{document}